\documentclass[11pt,a4paper,reqno]{article}

\usepackage{graphics}   
\usepackage{graphics}
\usepackage{mathdots}
\usepackage[left]{lineno}
\usepackage{blindtext}
\usepackage{color}
\usepackage{epsfig}
\usepackage{enumitem}
\usepackage{anysize}
\usepackage{latexsym,amssymb,amsmath,amscd,amsthm,amsxtra}
\marginsize{2.5cm}{2.5cm}{1.0cm}{1.0cm}

\newtheorem{theorem}{Theorem}

\newtheorem{lemma}{Lemma}
\newtheorem{corollary}[theorem]{Corollary}

\newtheorem{algorithm}[theorem]{Algorithm}
\newtheorem{example}[theorem]{Example}
\numberwithin{equation}{section}

\newtheorem{thm}{Theorem}[section]
\newcommand{\bth}{\begin{thm}}
\newtheorem{bdef}{Definition}
\newcommand{\brdef}{\begin{bdef}}
\newcommand{\erdef}{\end{bdef}}

\newtheorem{lem}{Lemma}[section]
\newcommand{\blem}{\begin{lem}}
\newcommand{\elem}{\end{lem}}

\newtheorem{rem}{Remark}
\newcommand{\brem}{\begin{rem}}
\newcommand{\erem}{\end{rem}}

\newtheorem{corr}{Corollary}[section]
\newcommand{\bcor}{\begin{corr}}
\newcommand{\ecor}{\end{corr}}

\newcommand{\bproof}{\begin{proof}}
\newcommand{\eproof}{\end{proof}}

\def\({\left(}
\def\){\right)}

\def\trace{\mathop{\mathrm{trace}}}
\def\Im{\mathop{\mathrm{Im}}}
\def\diag{\mathop{\mathrm{diag}}}
\def\max{\mathop{\mathrm{max}}}
\def\Re{\mathop{\mathrm{Re}}}
\def\cond{\mathop{\mathrm{cond}}}

\begin{document}

\title {\textrm{\textbf{Computation of matrix gamma function }}}

\author{\normalsize{\textbf{Jo\~{a}o R. Cardoso$^{a}$, Amir Sadeghi $^{b}$\footnote{Corresponding author (E-mail address: drsadeghi.iau@gmail.com)} }} \\
\small{\textit{$^{a}$  Coimbra Polytechnic -- ISEC, Coimbra, Portugal, and}}\\
\small{\textit{Institute of Systems and Robotics,
 University of Coimbra, P\'{o}lo II, Coimbra -- Portugal}}\\
\small{\textit{$^{b}$Department of Mathematics, Robat Karim Branch, Islamic Azad University, Tehran, Iran.    }}
}

\date{}
\maketitle
\vspace {-.5cm}

\maketitle
\thispagestyle{empty}

\begin{abstract}

	Matrix functions with potential applications have a major role in science and engineering. One of the fundamental matrix functions, which is particularly important due to its connections with certain matrix differential equations and other special matrix functions, is the matrix gamma  function. This research article is focused on the numerical computation of this  function. Well-known techniques for the scalar gamma function, such as Lanczos and Spouge methods, are carefully extended to the matrix case. This extension raises many challenging issues and several strategies used in the computation of matrix functions, like Schur decomposition and block Parlett recurrences, need to be incorporated to turn the methods more effective. We also propose a third technique based on the reciprocal gamma function that is shown to be competitive with the other two methods in terms of accuracy, with the advantage of being rich in matrix multiplications. Strengths and weaknesses of the proposed methods are illustrated with a set of numerical examples. Bounds for truncation errors and other bounds related with the matrix gamma function will be discussed as well.
\end{abstract}

\noindent \textit{keywords}: Gamma matrix function, Lanczos method, Matrix beta function, Spouge method, Reciprocal gamma function, Schur decomposition,
Block Parlett recurrence.


\section{Introduction}

Let $A\in \mathbb{C}^{n\times n}$ be a positive stable matrix (that is, $\Re(\lambda) > 0$, for all $\lambda\in\sigma(A)$). The gamma matrix function $\Gamma(A)$ may be defined by the convergent matrix improper integral \cite{Jodar2}
\begin{equation}\label{1-1}
\Gamma({A})=\int_{0}^{\infty}e^{-t}t^{{A}-{I}}dt.
\end{equation}
where $t^{A-I}:=\exp((A-I)\log t)$. Recall that if $z$ is a complex number not belonging to $\mathbb{R}_0^-$, we can define the ``scalar-matrix exponentiation'' $z^M$ as the function from $\mathbb{C}\times \mathbb{C}^{n\times n}$ to $\mathbb{C}^{n\times n}$ which assigns to each pair $(z,M)$ the $n\times n$ square complex matrix $z^M:=e^{M\log z}$, with $\log(z)$ standing for the principal logarithm. This function is a particular case of the more general ``matrix-matrix exponentiation'' addressed  recently in \cite{Cardoso}. 

It is well-known that the scalar gamma function is analytic everywhere in the complex plane, with the exception of non-positive integer numbers. Hence, the general theory of primary matrix functions \cite{Higham,Horn} ensures that $\Gamma(A)$ is well defined, provided that $A$ has no   eigenvalues being non-positive integer numbers.

Since the reciprocal gamma function, here denoted by $\Delta(z):=\frac{1}{\Gamma(z)}$, is an entire function, for any matrix $A\in\mathbb{C}^{n\times n}$, the $n\times n$ matrix $\Delta(A)=(\Gamma(A))^{-1}$ is a well defined matrix. Furthermore, if $A$ does not have any non-positive integer eigenvalue, then $A + mI$ is invertible, for all integer $m \geq 0$, and one gets the following formula \cite{Jodar2}, which uses the Pochhammer notation:
\begin{eqnarray*}
	(A)_{m}&=&A(A+I)\ldots(A+(m-1)I)\\
	&=&\Gamma(A+mI)\Delta(A),\qquad m\geq1
\end{eqnarray*}
and $(A)_{0}=I$. An alternative definition of the matrix gamma function, as a limit of a sequence of matrices, is provided in \cite{Jodar2}:
$$
\Gamma(A)=\lim_{m\rightarrow\infty}(m-1)!\left[ (A)_{m}\right]^{-1}m^{A}.
$$
We shall note that many definitions of the scalar gamma function (see, for instance, \cite[Ch. 6]{Abramowitz}) may be easily extended to the matrix case.

The matrix gamma function has connections with other special functions, which in turn play an important role to solving certain matrix differential equations; see \cite{Jodar2} and the references therein. Two of those special functions are the matrix beta and Bessel functions. If the matrices $A, B \in \mathbb{C}^{n\times n}$ satisfy the spectral conditions \cite{Defez}
$$\Re(z)>-1, \quad \forall z\in \sigma(A), \qquad \Re(z)<1, \quad \forall z \in \sigma(B),$$
then
$$
\int_{-1}^{1}(1+t)^{A}(1-t)^{B} dt=2^{A+I}\mathcal{B}(A+I,B+I)\,2^{B},
$$
where $\mathcal{B}(A,B)$ is the Beta matrix function \cite{Jodar2}, defined by
\begin{equation}\label{beta1}
\mathcal{B}(A,B)=\Gamma(A)\Gamma(B)\Delta(A+B),
\end{equation}
or
\begin{equation}\label{beta2}
\mathcal{B}(A,B)=\int_{0}^{1}t^{A-I}(1-t)^{B-I} dt.
\end{equation}
The matrix Bessel function can be defined by \cite{Jodar2,Sastre}:
$$
\mathcal{J}_{A}(z)=\sum_{k=0}^{\infty}\frac{(-1)^{k}\Delta\left(A+I/2\right)}{k!}\left(\frac{z}{2}\right)^{A+2kI}.
$$

There are several approaches to the computation of direct and reciprocal scalar gamma functions. To cite just a few, we mention the Lanczos approximation \cite{Lanczos,Pugh}, Spouge approximation \cite{Spouge,Pugh}, Stirling's formula \cite{Spira,Pugh}, continued fractions \cite{Char}, Taylor series \cite{Ahmed}, Schmelzer and Trefethen techniques \cite{Schmelzer} and Temme's formula \cite{Temme}. In his Ph.D thesis, Pugh \cite{Pugh} claims that the Lanczos approximation is the most feasible and accurate algorithm for approximating the scalar gamma function. If extended to matrices in a convenient way, we will see that, with respect to a compromise between efficiency and accuracy, Lanczos method can also be viewed as a serious candidate to the best method for the matrix gamma function. Another method for computing $\Gamma(A)$ that performs very well in terms of accuracy is based on a Taylor expansion of $\Delta(A)$ around the origin, combined with the reciprocal version of the so-called Gauss multiplication formula (check \cite[(6.1.20)]{Abramowitz}):
\begin{equation}\label{gauss-mult}
\Delta(z)=(2\pi)^{\frac{m-1}{2}}\,m^{\frac{1}{2}-z}\,\prod_{k=0}^{m-1}\Delta\left(\frac{z+k}{m}\right),
\end{equation}
where $m$ is a positive integer. The key point of this formula is that it exploits the fact that such a Taylor expansion is more accurate around the origin. Note that, for many values of $k$, $\frac{z+k}{m}$ in the right-hand side of (\ref{gauss-mult}) is closer to the origin than $z$. We also extend the Spouge method to matrices in Section \ref{spouge}. However, this extension gives poor results if we simply replace the scalar variable $z$ by $A$. The same holds for Lanczos and Taylor series methods. We must pay attention to some issues arising when dealing with matrices, namely the fact that a matrix may have simultaneously eigenvalues with positive and negative real parts. Our strategy has some similarities with the one used in \cite{Davies}, which includes, in particular, an initial Schur decomposition of $A$,
$$A=UTU^\ast,$$
with $U$ unitary and $T$ upper triangular, a reorganization of the diagonal entries of $T$ in blocks with ``close'' eigenvalues and a block Parlett recurrence. It is, in particular, important to ensure a separation between the eigenvalues with negative and positive real parts.

In contrast with other matrix functions, like the matrix square root, matrix exponential or the matrix logarithm, little attention has been paid to the numerical computation of the matrix gamma function. According to our knowledge, we are the first to investigate in depth the numerical computation of this function. Indeed, we have found in the literature only two marginal references to the numerical computation of the matrix gamma function. Schmelzer and Trefethen \cite{Schmelzer} mentioned that the Hankel's contour integral representation given by Eq. (2.1) in \cite{Schmelzer} can be generalized to square matrices $A$, and that their methods can be used to compute $\Delta(A)$. They claimed to have confirmed this by numerical experiments but no results are reported in their paper. They also stated that ``a drawback of such methods is that it is expensive to compute $s_{k}^{-A}$ for every node; methods based on the algorithms of Spouge and Lanczos might be more efficient''.  In \cite{Hale}, at the end of Section 2, Hale et. al. mentioned that their method for computing certain functions of matrices having eigenvalues on or close to the positive real axis can be applied to the gamma function of certain matrices and give an example with a diagonalizable matrix of order $2$.

\medskip {\bf Notation:} $\Gamma(.)$ and $\Delta(.)$, denote, respectively, the gamma and its reciprocal; $\sigma(A)$ denotes the spectrum of the matrix $A$; $\Re(z)$ is the real part of the complex number $z$; $A^\ast$ is the conjugate transpose of $A$; $\mathbb{Z}_0^-$ is the set of non-positive integers; $\diag(.)$ denotes a diagonal matrix; $\alpha(A):=\max\{\Re(\lambda):\ \lambda\in\sigma(A)\}$ is the spectral abscissa of $A$; $\gamma(A,r)$ and $\Gamma(A,r)$ stand to the incomplete gamma function and its complement, respectively; $\|.\|$ denotes any subordinate matrix norm, $\|.\|_p$ (with $p=1,2,\infty$) is a $p$-norm, and $\|.\|_F$ the Frobenius norm.

\medskip
The paper is organized as follows. In Section \ref{revisit} we revisit some properties of the scalar gamma function and recall some of the most well-known methods for its numerical computation. Section \ref{matrix-gamma} is focused on theoretical properties of the matrix gamma function and on the derivation of bounds for the norm of the matrix gamma and its perturbations. The extension of Lanczos and Spouge approximations to the matrix case is carried in Section \ref{strategies}, where a Taylor series expansion of the reciprocal gamma function is also proposed for computing the matrix gamma. To turn the approximation techniques more reliable, in Section \ref{schur-parlett} we show how to incorporate the Schur-Parlett method and point out its benefits. Numerical experiments are included in Section \ref{experiments} to illustrate the behaviour of the methods and some conclusions are drawn in Section \ref{conclusions}.

\section{Revisiting the Scalar Gamma and Related Functions}\label{revisit}

This section includes a brief revision on some topics related with the scalar gamma function that are relevant for the subsequent material. For readers interested in a more detailed revision, we suggest, among the vast literature, the works \cite[Ch. 6]{Abramowitz} and \cite{Davis}. Check also \cite{Borwein,Wiki1} and the references therein.

\subsection{Definition and Properties}\label{def-prop}

Among the many equivalent definitions to the scalar gamma function, the most used seems to be the following one defined, for a complex number $z$ with positive real part, via the convergent improper integral
\begin{equation}\label{gamma-def}
\Gamma({z})=\int_{0}^{\infty}e^{-t}t^{{z}-{1}}dt,\quad \Re(z)> 0.
\end{equation}
This integral function can be extended by analytic continuation to all complex numbers except the non-positive integers $z\in\mathbb{Z}_0^-=\{0,-1,-2,\ldots\}$, where the function has simple poles. Unless otherwise is stated, we shall assume throughout the paper that $z\notin \mathbb{Z}_0^-$.

Integrating (\ref{gamma-def}) by parts, yields the identity
\begin{equation}\label{gamma-id1}
\Gamma({z+1})=z\Gamma(z).
\end{equation}
Accounting that $\Gamma(1)=1$, a connection between the gamma function and factorials results easily:
$$\Gamma(m+1)=m!,$$
for any $m=0,1,2,\ldots$. 
Another important identity satisfied by the gamma function is the so-called reflection formula
\begin{equation}\label{gamma-id2}
\Gamma({z})=\frac{\pi}{\Gamma(1-z)\sin(\pi\,z)}, \quad z\notin \mathbb{Z},
\end{equation}
which is very useful in the computation of the gamma function on the left-half plane.

Related with this function are the {\it reciprocal gamma} function
$$\Delta(z):=\frac{1}{\Gamma(z)},$$
which is an entire function (here denoted by $\Delta(.)$ to avoid confusion with the inverse function of $\Gamma(.)$),
and the {\it incomplete gamma} function
\begin{equation}\label{gamma-inc}
\gamma(z,r):=\int_{0}^{r}e^{-t}t^{z-1}dt,\quad \Re(z)>0,\ r>0.
\end{equation}
Due to the amenable properties of the reciprocal gamma function, some authors have used it as a means for computing $\Gamma(z)$. Two reasons for this are:   $\Delta(z)$ can be represented by the Hankel integral \cite{Abramowitz,Trefethen06,Schmelzer}
\begin{equation}\label{hankel}
\Delta(z)=\frac{1}{2\pi i}\int_{\mathcal C} t^{-z}e^t\,dt,
\end{equation}
where the path ${\mathcal C}$ is a contour winding around the negative real axis in the anti-clockwise sense, and by the Taylor series with infinite radius of convergence \cite{Abramowitz,Wrench,Wrench73,Fekih}
\begin{equation}\label{reciprocal-series}
\Delta(z)=\sum_{k=0}^\infty a_kz^k,\quad |z|<\infty,
\end{equation}
where $a_1=1,\ a_2=\gamma$ (here $\gamma$ stands for the Euler-Mascheroni constant), and the coefficients $a_k$ ($k\geq 2$) are given recursively by \cite{Bourguet,Wrench}
\begin{equation}\label{reciprocal-coeff}
a_k=\frac{a_2a_{k-1}-\sum_{j=2}^{k-1}(-1)^j\zeta(j)a_{k-j}}{k-1},
\end{equation}
with $\zeta(.)$ being the Riemann zeta function. Approximations to $a_2,\ldots,a_{41}$ with 31 digits of accuracy are provided in \cite[Table 5]{Wrench}; see also \cite[p.256 (6.1.34)]{Abramowitz} and \cite{Bourguet}. 
New integral formulae, as well as asymptotic values, for $a_k$ have been recently proposed in \cite{Fekih}. By observing Figure \ref{figure-2}, which displays the graph of the reciprocal gamma as a real function with a real variable, large errors are expected when approximating $\Delta(x)$ by the series (\ref{reciprocal-series}) for values of $x$ with large magnitude. So a reasonable strategy is to combine (\ref{reciprocal-series}) with the Gauss formula (\ref{gauss-mult}). By choosing a suitable $m$, the magnitude of the argument $x$ is reduced and a truncation of (\ref{reciprocal-series}) is used to approximate $\Delta(x)$ with $x$ having small magnitude. Note that, as shown in Figure \ref{figure-2}, the values of $\Delta(x)$ for $x$ small are moderate. Note also that it is not always possible to put all the arguments $\frac{z+k}{m}$ in (\ref{gauss-mult}) very close to the origin, even for large $m$. Indeed, $\left| \frac{z+(m-1)}{m}\right|\geq 1$, for all complex $z$ with $\Re(z)\geq 1$ and any positive integer $m$.

\begin{figure}[ht]
	\centering
	\includegraphics[width=12cm]{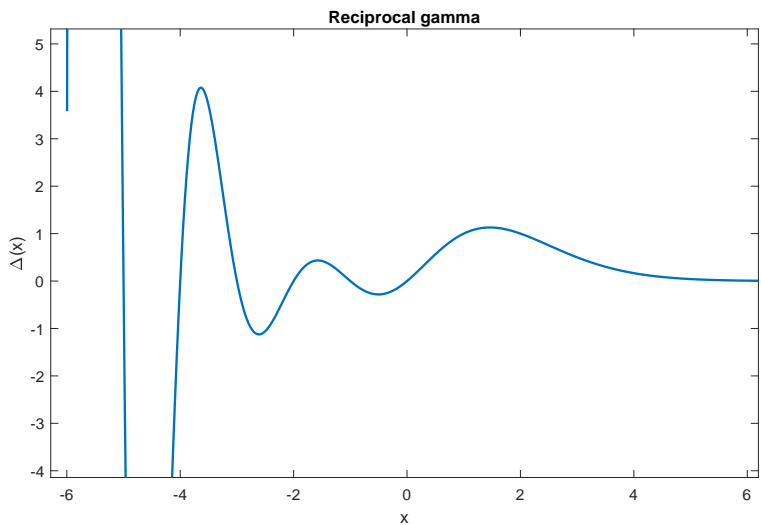}
	\caption{\small Graph of the reciprocal gamma function for real arguments. }
	\label{figure-2}
\end{figure}

Due to its important role in applications, in particular, in Statistics, the incomplete gamma function has attracted the interest of many researchers. For an overview, we suggest \cite{Wiki3}, paying attention to the large list of references therein; see also the book \cite{Chaudhry}. For the numerical computation see, for instance, \cite{Gautschi,Gautschi99,Smith01,Winitzki}.

Numerical quadrature can be applied to the integral (\ref{gamma-def}), but the results may not be satisfactory. For some complex numbers $z$, the magnitude of the integrand function may have a large variation. This happens in particular when $0<\Re(z)<1$. This phenomenon also occurs whenever $\Re(z)>1$, but the magnitude is not so large. To give more insight, let us consider the absolute value of the integrand function in (\ref{gamma-def}) as depending on two variables $t$ and $z$:
$$
f(t,z)=\left|e^{-t}t^{z-1}\right|.
$$
After a little calculation, $f$ simplifies to $f(t,z)=e^{-t}t^{\Re(z)-1}.$ Since the imaginary part of $z$ does not matter, one may view $f$ as a real function with two real variables $t$ and $x$: $\ f(t,x)=e^{-t}t^{x-1}.$ To illustrate the variation of $f$, we fix $t=2^{-52}$ (this is {\tt eps} in MATLAB) and consider three values to $x=\Re(z)$ (smaller, equal and greater to $1$):
$$
f(2^{-52},0.8)=1.35\times 10^3,\quad f(2^{-52},1)=1.00\times 10^0,\quad f(2^{-52},1.2)=7.40\times 10^{-4}.
$$
In addition, the graphs in Figure \ref{figure-1} illustrate the situation when $x=0.8,\,1,\,1.2,\,7$ and $t$ varies. We can observe, in particular, a high variation of $f$ for $x<1$ (top-left plot) and that the approximation of $f$ to zero is in a slower fashion as $x$ becomes larger (bottom-right plot). It is worth to recall that if the integrand function does not decay in a fast way, a truncation of the integral in (\ref{gamma-def}) may not work.

Notice that quadrature used in \cite{Schmelzer} is applied to an integral different from (\ref{gamma-def}). Other techniques like Lanczos and Spouge approximations are preferable (see next two subsections) than the evaluation of (\ref{gamma-def}) by numerical quadrature. The famous Stirling's formula has been often used, but it is not addressed here.

It should be stressed out, however, that the integral (\ref{gamma-def}) is very interesting from a theoretical viewpoint and will be useful in Section \ref{matrix-gamma} to derive scalar bounds for the norm of matrix gamma function.

\begin{figure}[ht]
	\centering
	\includegraphics[width=15cm]{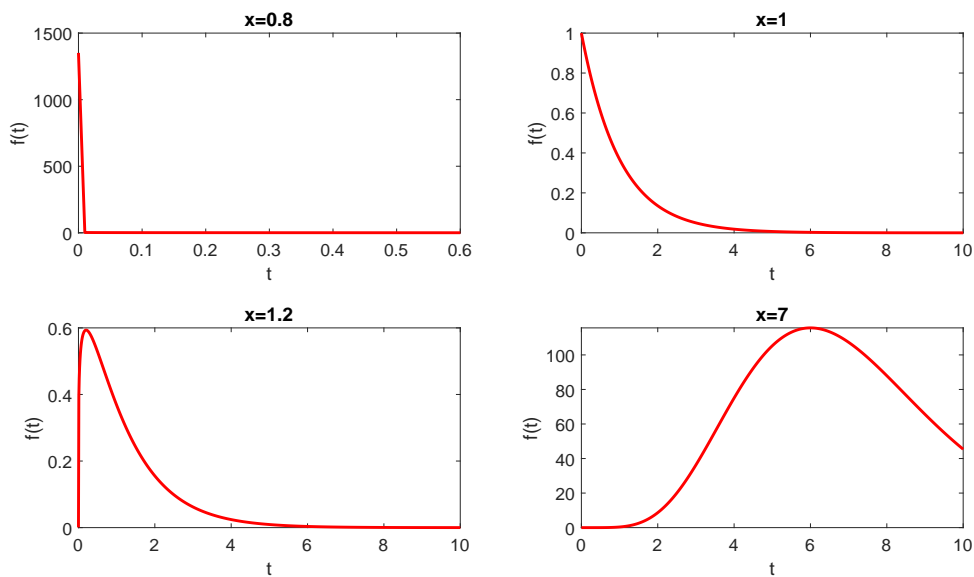}
	\caption{\small Graphs of the absolute value of the function under the integral (\ref{gamma-def}) for several values of $x=\Re(z)$. Note that the interval where $t$ varies is smaller in the top-left plot. }
	\label{figure-1}
\end{figure}

\subsection{Lanczos Approximation}\label{sec-lanczos}
In his remarkable paper \cite{Lanczos}, Lanczos start by showing that
$$\Gamma(z+1)=(z+\alpha+1)^{z+1}\,e^{-(z+\alpha+1)}\int_{0}^{e}\left(t(1-\log t)\right)^z t^\alpha\,dt,$$
where $\alpha$ is a fixed positive real number and $\Re(z)>-1$. Then he approximates the integral on the right-hand side by a partial fraction expansion
\begin{equation}\label{trunc}
\int_{0}^{e}\left(t(1-\log t)\right)^z t^\alpha\,dt \approx c_0+\sum_{k=1}^{m} \frac{c_k}{z+k},
\end{equation}
where $\alpha$ and $m$ are appropriately chosen in order to control either the truncation error of the approximation (\ref{trunc}) and the number of terms in the partial fraction expansion. The values of $c_k$ for some  parameters $\alpha$ are listed in \cite[p. 94]{Lanczos}. For instance, the choice $\alpha=5$ and $m=6$ guarantees a truncation error of at most $2\times 10^{-10}$ for all $z$ in the right-half plane. A more complete list of values of the coefficients $c_k$ is given in \cite[App. C]{Pugh} together with empirical estimates for the truncation error. For computations in IEEE double precision arithmetic, Pugh \cite{Pugh} recommends to use $\alpha=10.900511$ and $m=10$. However, in the implementation of the Lanczos method provided in \cite{Godfrey}, Godfrey uses $\alpha=9$ and $m=10$. He claims that such values for $\alpha$ and $m$ guarantee a relative error smaller than $10\times 10^{-13}$ for a large set of positive real numbers. He has also suggested a new method for computing the coefficients $c_k$, because the one used by Lanczos is rather complicated and sensitive to rounding errors. The behaviour of the values of those coefficients and the main issues raised by their computation  are discussed with detail in \cite{Pugh}.

For convenience, Lanczos formula is in general given in the form
\begin{equation}\label{lanczos-for}
\Gamma(z+1)=\sqrt{2\pi}(z+\alpha+1/2)^{z+1/2}\,e^{-(z+\alpha+1/2)}\left[c_0(\alpha)+
\sum_{k=1}^{m}\frac{c_k(\alpha)}{z+k}+\epsilon_{\alpha,m}(z)\right],
\end{equation}
where $\epsilon_{\alpha,m}(z)$ denotes the truncation error arising in (\ref{trunc}).
Often, to avoid overflow on (\ref{lanczos-for}), using the following logarithmic version and then exponentiate may be more practical:
\begin{eqnarray}
\log\left(\Gamma(z+1)\right)&=&\frac{1}{2}\log(2\pi)+(z+1/2)\log(z+\alpha+1/2)-(z+\alpha+1/2)+\nonumber\\
&&\log\left[c_0(\alpha)+
\sum_{k=1}^{m}\frac{c_k(\alpha)}{z+k}+\epsilon_{\alpha,m}(z)\right].\label{lanczos-for-log}
\end{eqnarray}

\bigskip

\subsection{Spouge Approximation}\label{spouge-app-scalar}

An improvement of the work of Lanczos was given in 1994 by Spouge in the paper \cite{Spouge}. There, the author proposes
the formula
\begin{equation}\label{spouge-for}
\Gamma(z)=\sqrt{2\pi}(z-1+a)^{z-1/2}e^{-(z-1+a)}\left[d_{0}(a)+\sum_{k=1}^{m}\frac{d_{k}(a)}{z-1+k}+e_a(z)\right]
\end{equation}
which is valid for $\Re(z-1+a)\geq 0$. The parameter $a$ is a positive real number, $m=\lceil a\rceil-1$ ($\lceil .\rceil$ denotes the ceil of a number),  $d_{0}=1$, and $d_{k}(a)$ is the residue of
$$
\Gamma(z)(z-1+a)^{-(z-1/2)}e^{z-1+a}(\sqrt{2\pi})^{-1}
$$
at $z=-k+1$. Explicitly, for $1\leq k \leq n$, 
$$
d_{k}(a)=\frac{1}{\sqrt{2\pi}}\frac{(-1)^{k-1}}{(k-1)!}(-k+a)^{k-0.5}e^{-k+a}.
$$
Spouge's formula has the very simple relative error bound
$$
|\varepsilon_a(z)|=\left|\frac{e_a(z)}{\Gamma(z)(z-1+a)^{-(z-1/2)}e^{z-1+a}(\sqrt{2\pi})^{-1}}\right|.
$$
In other words,
\begin{equation}\label{spouge-err-scalar}
|\varepsilon_a(z)|\leq \frac{\sqrt{a}}{(2\pi)^{a+1/2}}\frac{1}{\Re(z-1+a)},
\end{equation}
provided $a\geq3$. Thus for $z$ in the right half plane $\Re(z)\geq 0$, $|\varepsilon_a(z)|$ has the uniform bound
$$
|\varepsilon_a(z)|\leq \frac{1}{\sqrt{a}(2\pi)^{a+0.5}}.
$$
Note that $\varepsilon_a(z)$ and $e_a(z)$ are different type of errors. Moreover, in contrast with Lanczos, the concept of relative error used by Spouge coincides with the standard one, that is,
\begin{equation}\label{err-spouge-scalar}
\varepsilon_a(z)=\frac{\Gamma(z)-G_a(z)}{\Gamma(z)},
\end{equation}
where $G_a(z)$ is the approximation to gamma function obtained from Spouge formula:
$$
G_a(z):=\sqrt{2\pi}(z-1+a)^{z-1/2}e^{-(z-1+a)}\left[d_{0}(a)+\sum_{k=1}^{m}\frac{d_{k}(a)}{z-1+k}\right].
$$

When compared with the Lanczos formula, Spouge formula has the advantages of having simpler error estimates and its coefficients $d_k$ are much
easier to compute. Note, however, that in general Spouge formula is not so accurate as
the one of Lanczos for the same number of terms in the series. 

\section{Matrix Gamma Function}\label{matrix-gamma}

To our knowledge, the first investigations on the matrix gamma function in a detailed form were carried out in \cite{Jodar2}. There, the authors propose
the definition (\ref{1-1}) via a convergent matrix improper integral, and derive some properties. More properties are addressed in   \cite{Jodar1,Cortes}. In those papers, matrix gamma function has been investigated in a strict connection with the matrix beta function.
This is due to the relationships mentioned in (\ref{beta1}) and (\ref{beta2}).

As recalled in Section \ref{def-prop}, $\Gamma(z)$ is analytic in the complex plane with the exception of non-positive integer numbers. From the well-known theory of matrix functions \cite{Higham,Horn}, many properties of the scalar gamma function can be easily extended to the matrix scenario, provided that some restrictions on the spectrum of matrices are considered. In Lemma \ref{lema-basic} below, some of those properties are stated, especially the ones that are needed later in the paper. New bounds for the matrix gamma function and its perturbations are also proposed in this section. We believe they will contribute to understand better the numerical behaviour of this function.

\subsection{Basic Properties}

\begin{lemma}\label{lema-basic}
	Let ${A}\in \mathbb{C}^{n\times n}$ has no eigenvalues on $\mathbb{Z}_0^-$. Then the following properties hold:
	\begin{enumerate}
		\item[(i)] $\Gamma(I)=I$, and $\Gamma(A+I)=A\Gamma(A)$; \\
		\vspace{-.5cm}
		\item[(ii)] $\Gamma(A)$ is nonsingular;   \\
		\vspace{-.5cm}
		\item[(iii)] If $A$ is block diagonal, say $A=\diag(A_{1},\ldots,A_{m})$, then $\Gamma(A)$ is a block diagonal matrix with the same block structure, that is, $\Gamma(A)=\diag(\Gamma(A_{1}),\ldots,\Gamma(A_{m}))$;\\
		\vspace{-.5cm}
		\item[(iv)] $\Gamma(A^\ast)=\Gamma(A)^\ast$;\\
		\vspace{-.5cm}
		\item[(v)] If there is a nonsingular complex matrix $S$ and a complex matrix $B$ such that $A=SBS^{-1}$, then $\Gamma(A)=S\,\Gamma(B)\,S^{-1}$; \\
		\vspace{-.5cm}
		\item[(vi)] Assuming in addition that $A$ does not have any integer eigenvalue, one has the matrix reflection formula 
		\begin{equation}\label{reflection}
		\Gamma(A)\Gamma(I-A)=\pi\,\left[\sin(\pi A)\right]^{-1}.
		\end{equation}
	\end{enumerate}
\end{lemma}
\begin{proof}
	All the statements follow easily from the theory of matrix functions. For further information on the matrix sine function arising in (\ref{reflection}), check \cite[Ch. 12]{Higham}. \qed
\end{proof}

Next example shows that closed expressions for the matrix gamma function can be very complicated, even to diagonalizable matrices of order $2$.

\begin{example} 
	{\rm 
		Consider the diagonalizable matrix $A=\bigl[\begin{smallmatrix}
		a & b \\
		c & d \\
		\end{smallmatrix} \bigr]$, which has eigenvalues $\lambda_{1}=\frac{(a+d)- \Omega}{2}$ and $\lambda_{2}=\frac{(a+d)+ \Omega}{2}$, where $\Omega=\sqrt{(a-d)^{2}+4bc}$ is assumed to be non zero. The eigenvalues and eigenvectors of $A$ are
		$$D=\textrm{diag}(\lambda_{_{1}},\lambda_{2})=\begin{pmatrix}
		\frac{(a+d)- \Omega}{2} & 0 \\
		0 & \frac{(a+d)+ \Omega}{2} \\
		\end{pmatrix}, \quad X=\begin{pmatrix}
		-\frac{(a-d)+ \Omega}{2c} & \frac{(a-d)+ \Omega}{2c} \\
		1 & 1 \\
		\end{pmatrix}.
		$$
		Hence, $\Gamma(A)$ can be evaluated by the spectral decomposition $\Gamma(A)=X\Gamma(D)X^{-1}$ as following:
		\begin{eqnarray*}
			\Gamma(A)&=&\begin{pmatrix}
				-\frac{(a-d)+ \Omega}{2c} & \frac{(a-d)+ \Omega}{2c} \\
				1 & 1 \\
			\end{pmatrix}\begin{pmatrix}
				\Gamma(\lambda_{1}) & 0 \\
				0 & \Gamma(\lambda_{2}) \\
			\end{pmatrix}\begin{pmatrix}
				-\frac{(a-d)+ \Omega}{2c} & \frac{(a-d)+ \Omega}{2c} \\
				1 & 1 \\
			\end{pmatrix}^{-1} \\
			&=& \frac{1}{2\Omega}\small{\begin{pmatrix}
					\Gamma(\lambda_{1})(d-a+\Omega)+\Gamma(\lambda_{2})(a-d+\Omega) & -2b(\Gamma(\lambda_{1})-\Gamma(\lambda_{2})) \\
					-2c(\Gamma(\lambda_{1})-\Gamma(\lambda_{2})) & \Gamma(\lambda_{1})(a-d+\Omega)+\Gamma(\lambda_{2})(d-a+\Omega) \\
			\end{pmatrix}}.
		\end{eqnarray*}
	}
\end{example}

\subsection{Norm Bounds}
Before proceeding with investigations on bounding the norm of the matrix gamma function, we shall recall that the {\it incomplete gamma} function with a matrix argument can be defined by \cite{Sastre}
$$
\gamma(A,r):=\int_{0}^{r}e^{-t}t^{A-I}dt,
$$
and its {\it complement} by
$$
\Gamma(A,r):=\int_{r}^{\infty}e^{-t}t^{A-I}dt,
$$
where it is assumed that $A\in \mathbb{C}^{n\times n}$ satisfies $\Re(\lambda)>0$, for all $\lambda\in\sigma(A)$, and $r$ is a positive real number. We also remind the definition and notation to the {\it spectral abscissa} of $A$:
\begin{equation}\label{alpha_A}
\alpha(A):=\max\{\Re(\lambda):\ \lambda\in\sigma(A)\}.
\end{equation}

\begin{theorem}\label{theorem1}
	Given $A\in \mathbb{C}^{n\times n}$ satisfying $\Re(\lambda)>0$, for all $\lambda\in\sigma(A)$, let
	\begin{equation}\label{N}
	A=U(D+N)U^\ast
	\end{equation}
	be its Schur decomposition, with $U$, $D$ and $N$ being, respectively, unitary, diagonal and strictly upper triangular matrices.
	If $r\geq 1$, then the complement of gamma function allows the following bound, with respect to the $2$-norm:
	\begin{equation}\label{bound-inc}
	\|\Gamma(A,r)\|_2\leq \sum_{k=0}^{n-1} \frac{\|N-I\|_2}{k!}\,\Gamma\left(\alpha(A)+k,r\right),
	\end{equation}
	where $\alpha(A)$ is the spectral abscissa of $A$.
\end{theorem}

\begin{proof}
	The bound to the matrix exponential proposed in \cite[(2.11)]{VanLoan} states that
	\begin{equation}\label{bound-vanloan}
	\left\|e^{At}\right\|_2\leq e^{\alpha(A)t}\sum_{k=0}^{n-1} \frac{\|Nt\|_2^k}{k!}.
	\end{equation}
	(Notice that there is a typo in \cite[(2.11)]{VanLoan}; check \cite[Thm. 10.12]{Higham}).
	From (\ref{bound-vanloan}), and attending that $|\log(t)|\leq t$, for all $t\geq 1$, it follows easily that
	\begin{eqnarray*}
		\left\|t^{A-I}\right\| & \leq & e^{\alpha(A-I)\log(t)} \sum_{k=0}^{n-1} \frac{|\log(t)|^k\|(N-I)\|_2^k}{k!} \\
		& \leq & t^{\alpha(A)-1}\sum_{k=0}^{n-1} \frac{t^k\|(N-I)\|_2^k}{k!}.
	\end{eqnarray*}
	Hence, for $r\geq 1$,
	\begin{eqnarray*}
		\|\Gamma(A,r)\|_2 & = & \left\|\int_r^\infty e^{-t}t^{A-I}\ dt\right\|_2 \\
		& \leq & \int_r^\infty e^{-t}\left\|t^{A-I}\right\|_2\ dt \\
		& \leq & \int_r^\infty e^{-t}\,t^{\alpha(A)-1}\sum_{k=0}^{n-1} \frac{t^k\|(N-I)\|_2^k}{k!} \ dt \\
		& = & \sum_{k=0}^{n-1} \frac{\|(N-I)\|_2^k}{k!} \int_r^\infty e^{-t} t^{\alpha(A)-1}t^k\ dt \\
		& = & \sum_{k=0}^{n-1} \frac{\|(N-I)\|_2^k}{k!}\, \Gamma\left(\alpha(A)+k,r\right).
	\end{eqnarray*} \qed
\end{proof}

The previous theorem gives a scalar upper bound for the error arising in the approximation of the matrix gamma function by the matrix incomplete  gamma function. Indeed, since
$$
\Gamma(A)=\gamma(A,r)+\Gamma(A,r),
$$
for any $r>0$, the error of the approximation
$$
\Gamma(A)\approx \gamma(A,r),
$$
with $r\geq 1$, is bounded by (\ref{bound-inc}). At this stage, we may ask why using such an expensive upper bound involving the Schur decomposition of $A$ (its computation requires about $25n^3$ flops) instead of a cheaper one. The reason is partially explained in the paper \cite{VanLoan}. In fact,  there are many cheaper bounds to the matrix exponential than \ref{bound-vanloan}, but they are not sharp in general. Other reason is that our algorithms to be proposed later are based on the Schur decomposition and so bound (\ref{bound-inc}) can be computed at a negligible cost.
Next result provides an upper bound to the norm of the matrix gamma function.

\begin{corollary}
	Assume that the assumptions of Theorem \ref{theorem1} are valid. Then
	$$
	\|\Gamma(A)\|_2 \leq \sum_{k=0}^{n-1} \frac{\|N-I\|_2}{k!}\,\left[\gamma\left(\alpha(A)-k,1\right)+\Gamma\left(\alpha(A)+k,1\right)\right].
	$$
\end{corollary}

\begin{proof}
	Accounting that
	$$
	\Gamma(A)=\gamma(A,1)+\Gamma(A,1),
	$$
	by Theorem \ref{theorem1}, one just needs to show that
	$$\left\|\gamma(A,1)\right\|_2 \leq \sum_{k=0}^{n-1} \frac{\|N-I\|_2}{k!}\,\gamma\left(\alpha(A)-k,1\right).$$
	This result follows if we use the inequality
	\begin{equation}\label{ine-log}
	|\log(t)|\leq t^{-1},
	\end{equation}
	where $0<t\leq 1$, and the same strategy of the proof of Theorem \ref{theorem1}. Notice that (\ref{ine-log}) does not hold for $t=0$. However, this is not a problem because $\lim_{t\rightarrow 0^+} e^{-t}t^{A-I}=0$. \qed
\end{proof}

We end this section with a perturbation bound to the matrix gamma function. Now the norm can be an arbitrary subordinate matrix norm.

\begin{theorem}
	Let $A,E\in \mathbb{C}^{n\times n}$ and assume that the eigenvalues of $A$ and $A+E$ have positive real parts. Then
	$$
	\|\Gamma(A+E)-\Gamma(A)\| \leq \|E\|\left(\gamma(-\mu+1,1)+\Gamma(\mu+1,1)\right),
	$$
	where $\mu:=\max\{\|A+E-I\|,\,\|A-I\|\}.$
\end{theorem}

\begin{proof}
	From \cite[Thm. 5.1]{Cardoso}, a simple calculation shows that the following inequality holds for any subordinate matrix norm:
	\begin{equation}\label{perturb1}
	\left\|t^{A+E-I}-t^{A-I}\right\|\leq \|E\|\,e^{|\log(t)|\,\mu},
	\end{equation}
	with $\mu:=\max\{\|A+E-I\|,\,\|A-I\|\}.$ Hence
	\begin{equation}\label{perturb1}
	\left\|t^{A+E-I}-t^{A-I}\right\|\leq
	\left\{
	\begin{array}{lcl}
	\|E\|\,t^\mu,&\ \mbox{if}&\ t\geq 1\\
	\|E\|\,t^{-\mu},&\ \mbox{if}&\ 0<t<1
	\end{array}
	\right..
	\end{equation}
	Since
	$$
	\Gamma(A+E)-\Gamma(A)=\int_0^1 e^{-t}\left(t^{A+E-I}-t^{A-I}\right)\ dt +  \int_1^\infty e^{-t}\left(t^{A+E-I}-t^{A-I}\right)\ dt,
	$$
	the result follows by taking norms and attending to (\ref{perturb1}). \qed
\end{proof}

\section{Strategies for Approximating the Matrix Gamma Function}\label{strategies}

This section is devoted to the numerical computation of the matrix gamma function. We start by extending the well-known scalar methods of Lanczos and Spouge to matrices. A method based on the reciprocal gamma function used in combination with the Gauss multiplication formula is also addressed.

\subsection{Lanczos Method}

Before stating the Lanczos formula for matrices, we shall recall the concept of matrix-matrix exponentiation \cite{Barradas,Cardoso}.

If $A$ is an $n\times n$ square complex matrix with no eigenvalues on the closed negative real axis $\mathbb{R}_0^-$ and $B$ is an arbitrary square complex matrix of order $n$, the matrix-matrix exponentiation $A^B$ is defined as
\begin{equation}\label{AB}
A^B:=e^{\log({A}){B}},
\end{equation}
where $e^X$ stands for the exponential of the matrix $X$ and $\log(A)$ denotes the principal logarithm of $A$, i.e., the unique solution of the matrix equation $e^X=A$ whose eigenvalues lie on the open strip $\{ z\in \mathbb{C} :-\pi <\Im z<\pi\}$ of the complex plane; $\Im z$ stands for the imaginary part of $z$. For background on matrix exponential and matrix logarithm see \cite{Higham,Horn} and the references therein. Regarding the computation of matrix exponential and logarithm in the recent versions of MATLAB, the function \texttt{expm} implements the algorithm provided in \cite{Mohy09} and \texttt{logm} computes the matrix logarithm using an algorithm investigated in \cite{Mohy12,Mohy13}, which is an improved version of the inverse scaling and squaring with Pad\'e approximants method proposed in \cite{Kenney}. If $A$ has any eigenvalue on the negative real axis, \texttt{logm} computes a non principal logarithm. To avoid this situation, our investigations here will deal only with the computation of logarithms of matrices with no eigenvalues on the closed negative real axis $\mathbb{R}_0^-$.

Assuming that $A$ is an $n\times n$ matrix with all of its eigenvalues having positive real parts and $\alpha>0$, the matrix version of Lanczos formula (\ref{lanczos-for}) can be written as
\begin{eqnarray}
\Gamma(A)&=&\sqrt{2\pi}\left(A+(\alpha-0.5)I\right)^{A-0.5\,I}\,e^{-\left(A+(\alpha-0.5)I\right)}\times \nonumber\\
&&\left[c_0(\alpha)I+
\sum_{k=1}^{m} c_k(\alpha)\left(A+(k-1)I\right)^{-1}+e_{\alpha,m}(A)\right],\label{lanczos-for-matrix}
\end{eqnarray}
where $c_k(\alpha)$ are the Lanczos coefficients, which depend on the parameter $\alpha$. Discarding the error term $e_{\alpha,m}(A)$ in the right-hand side of (\ref{lanczos-for-matrix}), yields the approximation
\begin{eqnarray}
\Gamma(A)& \approx & \sqrt{2\pi}\left(A+(\alpha-0.5)I\right)^{A-0.5\,I}\,e^{-\left(A+(\alpha-0.5)I\right)} \times \nonumber\\
&&\left[c_0(\alpha)I+
\sum_{k=1}^{m} c_k(\alpha)\left(A+(k-1)I\right)^{-1}\right].\label{lanczos-approx-matrix}
\end{eqnarray}
A major difficulty of Lanczos method in the scalar case has been to bound the truncation error arising in the approximation (\ref{lanczos-for}). A good bound to that error would be useful to find optimal values to $m$ and $\alpha$. It turns out however that empirical strategies to find a good compromise between $m$, $\alpha$ and the magnitude of the error term $e_{\alpha,m}(A)$ have been used successfully in the implementation of Lanczos formula. This issue has been addressed in \cite{Lanczos} and a very interesting discussion has been carried out in \cite{Pugh}. Attending to our discussion in Section \ref{sec-lanczos}, in Algorithm \ref{alg-lanczos-1} below we will consider the values $m=10$ and $\alpha=9$ suggested in \cite{Godfrey}. The corresponding values for the coefficients $c_k(9)$, with $22$ digits of accuracy, are given in Table \ref{table-coeff}.

\begin{table}
	\begin{center}
		\begin{tabular}{cr} \hline
			$k$ & $c_k(9)$\\
			\hline
			$0$ & $1.000000000000000174663$ \\
			$1$ & $5716.400188274341379136$ \\
			$2$ & $-14815.30426768413909044$ \\
			$3$ & $14291.49277657478554025$ \\
			$4$ & $-6348.160217641458813289$ \\
			$5$ & $1301.608286058321874105$ \\
			$6$ & $-108.1767053514369634679$ \\
			$7$ & $2.605696505611755827729$ \\
			$8$ & $-0.7423452510201416151527\times 10^{-2}$ \\
			$9$ & $0.5384136432509564062961\times 10^{-7}$ \\
			$10$ & $-0.4023533141268236372067\times 10^{-8}$ \\
			\hline
		\end{tabular}
		
		\caption{Coefficients $c_k(\alpha)$ ($k=0,1,\ldots,10$) in the Lanczos formula (\ref{lanczos-for-matrix}) for $\alpha=9$, approximated with $22$ digits of accuracy.  }
		\label{table-coeff}
\end{center}\end{table}

\begin{algorithm}\label{alg-lanczos-1}
	{\rm This algorithm implements the Lanczos formula (\ref{lanczos-for-matrix}), with $\alpha=9$ and $m=10$, for approximating $\Gamma(A)$ , where the spectrum of $A\in \mathbb{C}^{n\times n}$ lies on the open right-half plane. Coefficients $c_k$ are given in Table \ref{table-coeff}.
		\begin{enumerate}
			\item Set $\alpha=9$, $m=10$ and $S=c_0I+c_1A^{-1}$;
			\item \texttt{for} $k=2:10$
			\item $\quad S=S+c_k(A+(k-1)I)^{-1}$;
			\item \texttt{end}
			\item $L=0.5\log(2\pi)I+(A-0.5I)\log(A+8.5\,I)-(A+8.5\,I)+\log(S)$;
			\item $\Gamma(A)\approx e^L$.
	\end{enumerate}}
\end{algorithm}

To avoid overflow, Algorithm \ref{alg-lanczos-1} uses the logarithmic version of Lanczos formula; check (\ref{lanczos-for-log}).

If the matrix $A$ has any eigenvalue with non positive real parts, Lanczos formula (\ref{lanczos-for-matrix}) may not be suitable for approximating $\Gamma(A)$. If all the eigenvalues of $A$ have negative real parts, one needs to use conveniently the reflection formula (\ref{reflection}), as will be given in Algorithm \ref{alg-lanczos-2}. In the more general case of $A$ having simultaneously eigenvalues with positive and negative real parts, Lanczos formula needs to be combined with a strategy separating the eigenvalues lying on the left-half plane with the ones in the right-half plane. This will be carried out in Section \ref{schur-parlett} by means of the so called Schur-Parlett method.

\begin{algorithm}\label{alg-lanczos-2}
	{\rm This algorithm implements the Lanczos formula (\ref{lanczos-for-matrix}), with $\alpha=9$ and $m=10$, for approximating $\Gamma(A)$ , where $A\in \mathbb{C}^{n\times n}$ is a  matrix with spectrum satisfying one and only one of the following conditions: (i) $\sigma(A)$ is contained in the open right-half plane; or (ii) $\sigma(A)$ does not contain negative integers and lies on the open left-half plane.
		\begin{enumerate}
			\item \texttt{if} $\Re(\trace(A))\geq 0$
			\item $\quad$Compute $\Gamma(A)$ by Algorithm \ref{alg-lanczos-1};
			\item \texttt{else}
			\item $\quad S=\sin(\pi A)$;
			\item $\quad$ Compute $G=\Gamma(I-A)$ by Algorithm \ref{alg-lanczos-1};
			\item $\quad$   $\Gamma(A)\approx \pi (SG)^{-1}$;
			\item \texttt{end}
	\end{enumerate}}
\end{algorithm}

\subsection{Spouge Method}\label{spouge}

Let $A\in\mathbb{C}^{n\times n}$ be a matrix having all eigenvalues with positive real parts and $a>0$. The matrix version of Spouge formula (\ref{spouge-for}) is:
\begin{eqnarray}
\Gamma(A)&=&\sqrt{2\pi}\left(A+(\alpha-0.5)I\right)^{A-0.5\,I}\,e^{-\left(A+(\alpha-0.5)I\right)}\times \nonumber\\
&&\left[d_0(a)I+
\sum_{k=1}^{m} d_k(a)\left(A+(k-1)I\right)^{-1}+e_{a}(A)\right],\label{spouge-for-matrix}
\end{eqnarray}
where $d_k(a)$ are the Spouge coefficients, which vary with $a$, and $m=\lceil a\rceil-1$. Ignoring the error term $e_{a}(A)$ in the right-hand side of (\ref{spouge-for-matrix}), we have
\begin{eqnarray}
\Gamma(A)&\approx & \sqrt{2\pi}\left(A+(\alpha-0.5)I\right)^{A-0.5\,I}\,e^{-\left(A+(\alpha-0.5)I\right)}\times \nonumber\\
&&\left[d_0(a)I+
\sum_{k=1}^{m} d_k(a)\left(A+(k-1)I\right)^{-1}\right].\label{spouge-approx-matrix}
\end{eqnarray}
Let us denote the relative truncation error of the approximation (\ref{spouge-approx-matrix}) by
\begin{equation}\label{spouge-err}
{\mathcal E}_a(A):=\frac{\|\Gamma(A)-G_a(A)\|}{\|\Gamma(A)\|},
\end{equation}
where $G_a(A)$ denotes the right-hand side of (\ref{spouge-approx-matrix}); see also Section \ref{spouge-app-scalar}. Note that ${\mathcal E}_a(A)$ does not result from the extension of error function (\ref{err-spouge-scalar}) to matrices and in general
$$
{\mathcal E}_a(A)\neq \|\varepsilon_a(A)\|=\|(\Gamma(A)-G_a(A))(\Gamma(A))^{-1}\|.
$$
Let $\kappa_p(X):=\|X\|_p\|X^{-1}\|_p$ denotes the condition number of the matrix $X$ with respect to the $p$-norm, with $p=1,2,\infty$. Next lemma gives a bound for the relative error ${\mathcal E}_a(A)$ with respect to $p$-norms for the case when $A$ is diagonalizable.

\begin{lemma}
	Let $A\in\mathbb{C}^{n\times n}$ be a diagonalizable matrix ($A=PDP^{-1}$, with $P$ nonsingular and $D:=\diag(\lambda_1,\ldots,\lambda_n)$) having all eigenvalues with positive real parts, that is, $\widetilde{\alpha}(A):=\min\{\Re(\lambda):\ \lambda\in\sigma(A)\}$ satisfies $\widetilde{\alpha}(A)> 0$. For $a\geq 3$ and ${\mathcal E}_a(A)$ given as in (\ref{spouge-err}), 
	\begin{equation}\label{bound-matrix-rel}
	{\mathcal E}_a(A) \leq \kappa_p(P) \frac{\sqrt{a}}{(2\pi)^{a-1/2}\left(\widetilde{\alpha}(A)-1+a\right)}.
	\end{equation}
\end{lemma}

\begin{proof}
	For any $z$ in the open right-half plane, we know, from Section \ref{spouge-app-scalar},  that
	\begin{equation}\label{gamma-g-scalar}
	\Gamma(z)-G_a(z)=\varepsilon_a(z)\Gamma(z).
	\end{equation}
	where $\varepsilon_a(z)$ is defined by (\ref{err-spouge-scalar}). Since $A$ has all the eigenvalues with positive real parts
	and the functions involved in (\ref{gamma-g-scalar}) are analytic on the right half-plane, the identity
	$$
	\Gamma(A)-G_a(A)=\varepsilon_a(A)\Gamma(A)
	$$
	is valid. Now, because $A$ is diagonalizable,
	$$
	\Gamma(A)-G_a(A)=P\varepsilon_a(D)P^{-1}\Gamma(A).
	$$
	Hence, for $p$-norms, we have
	\begin{eqnarray*}
		\|\Gamma(A)-G_a(A)\|_p & \leq & \kappa_p(P)\|\varepsilon_a(D)\|_p \|\Gamma(A)\|_p \\
		& \leq & \kappa_p(P) \max_{i=1,\ldots,n}|\varepsilon_a(\lambda_i)| \|\Gamma(A)\|_p,
	\end{eqnarray*}
	and, consequently,
	$$
	\frac{\|\Gamma(A)-G_a(A)\|_p}{\|\Gamma(A)\|_p} \leq \kappa_p(P) \max_{i=1,\ldots,n}|\varepsilon_a(\lambda_i)|.
	$$
	Therefore, the inequality (\ref{bound-matrix-rel}) follows from the Spouge scalar error bound (\ref{spouge-err-scalar}). \qed
\end{proof}

For a general matrix $A$ (diagonalizable or not), assume that the function  $E_a(z)=\Gamma(z)-G_a(z)$ (absolute error) is analytic on a closed convex set $\Omega$ containing the spectrum of $A$. A direct application of \cite[Thm. 4.28]{Higham} (check also \cite[Thm. 9.2.2]{Golub}), yields the bound (with respect to Frobenius norm)
\begin{equation}\label{bound-abs-err}
\|E_a(A)\|_F \leq \max_{i\leq k\leq n-1} \frac{\omega_k}{k!} \|(I-|N|)^{-1}\|_F,
\end{equation}
where $U^\ast AU=T=\diag(\lambda_1,\ldots, \lambda_n)+N$ is the Schur decomposition of $A$, with $T$ upper triangular, $N$ strictly upper triangular, and $\omega_k=\sup_{z\in\Omega}|E_a^{(k)}(z)|$. One drawback of bound (\ref{bound-abs-err}) is the need of the derivatives of
$E_a(z)$ up to order $n-1$.

Providing that $A$ is diagonalizable ($A=SDS^{-1}$) with $S$ not having a large condition number, the choice $a=12.5$ (and hence $m=12$) seems to be reasonable for working in IEEE double precision environments. This has been confirmed by many numerical experiments (not reported here) we have carried out. $a=12.5$ is also the value considered in \cite{Pugh} for scalars.

The corresponding Spouge coefficients are given in Table \ref{table-coeff-spouge} and the algorithms are presented below.

\begin{table}
	\begin{center}
		\begin{tabular}{cr} \hline
			$k$ & $d_k(12.5)$\\
			\hline
			$0$ & $1$ \\
			$1$ & $133550.5029424774402287$ \\
			$2$ & $-492930.9352993603097275$ \\
			$3$ & $741287.4736976117128506$ \\
			$4$ & $-585097.3776039966614917$ \\
			$5$ & $260425.2703303852758836$ \\
			$6$ & $-65413.35339611420204164$ \\
			$7$ & $8801.459635084211186040$ \\
			$8$ & $-564.8050241289801078892$ \\
			$9$ & $13.803798339181415855137$ \\
			$10$ & $-0.8078176169895076585981\times 10^{-1}$ \\
			$11$ & $0.3479741445742458983261\times 10^{-4}$ \\
			$12$ & $ -0.5689271227504240383584\times 10^{-11} $ \\
			\hline
		\end{tabular}
		
		\caption{Coefficients $d_k(a)$ ($k=0,1,\ldots,12$) in the Spouge formula (\ref{spouge-for-matrix}) for $a=12.5$, approximated with $22$ digits of accuracy.  }
		\label{table-coeff-spouge}
\end{center}\end{table}

\begin{algorithm}\label{alg-spouge-1}
	{\rm This algorithm implements the Spouge formula (\ref{spouge-for-matrix}), with $a=12.5$ and $m=12$, for approximating $\Gamma(A)$ , where $A\in \mathbb{C}^{n\times n}$ is a matrix with all eigenvalues lying on the open right-half plane. Coefficients $d_k$ are given in Table \ref{table-coeff-spouge}.
		\begin{enumerate}
			\item Set $\alpha=12.5$, $m=12$ and $S=d_0I+d_1A^{-1}$;
			\item \texttt{for} $k=2:12$
			\item $\quad S=S+d_k(A+(k-1)I)^{-1}$;
			\item \texttt{end}
			\item $L=0.5\log(2\pi)I+(A-0.5I)\log(A+11.5\,I)-(A+11.5\,I)+\log(S)$;
			\item $\Gamma(A)\approx e^L$.
	\end{enumerate}}
\end{algorithm}

\begin{algorithm}\label{alg-spouge-2}
	{\rm This algorithm implements the Spouge formula (\ref{spouge-for-matrix}), with $a=12.5$ and $m=12$, for approximating $\Gamma(A)$ , where $A\in \mathbb{C}^{n\times n}$ is a nonsingular matrix with spectrum satisfying one and only one of the following conditions: (i) $\sigma(A)$ is contained in the closed right-half plane; or (ii) $\sigma(A)$ does not contain negative integers and lies on the open left-half plane.
		\begin{enumerate}
			\item \texttt{if} $\Re(\trace(A))\geq 0$
			\item $\quad$Compute $\Gamma(A)$ by Algorithm \ref{alg-spouge-1};
			\item \texttt{else}
			\item $\quad S=\sin(\pi A)$;
			\item $\quad$ Compute $G=\Gamma(I-A)$ by Algorithm \ref{alg-spouge-1};
			\item $\quad$   $\Gamma(A)\approx \pi (SG)^{-1}$;
			\item \texttt{end}
	\end{enumerate}}
\end{algorithm}

\subsection{Reciprocal Gamma Function}

For any matrix $A\in \mathbb{C}^{n\times n}$, the reciprocal matrix gamma function allows the following Taylor expansion around the origin: 
\begin{equation}\label{reciprocal-series-matrix}
\Delta(A)=(\Gamma(A))^{-1}=\sum_{k=0}^\infty a_kA^k,
\end{equation}
where $a_k$ can be evaluated through the recursive formula (\ref{reciprocal-coeff}). 
According to our discussion in Section \ref{def-prop}, truncating (\ref{reciprocal-series-matrix}) to approximate $\Delta(A)$ is recommended only when the spectral radius of $A$ is small. If $A$ has a large spectral radius, then it is advisable to combine  (\ref{reciprocal-series-matrix}) with Gauss formula (\ref{gauss-mult}). 

For matrices having small norm ($\|A\|\leq 1$), next result proposes a bound for the truncation error of (\ref{reciprocal-series-matrix}) in terms of a scalar convergent series.

\begin{lemma}
	If  $A\in \mathbb{C}^{n\times n}$ with $\|A\|\leq 1$ and $a_k$ are the coefficients in (\ref{reciprocal-series-matrix}), then
	\begin{equation}\label{trunc-bound}
	\left\|\Delta(A)-\sum_{k=1}^m a_k A_k\right\|\lesssim \frac{4}{\pi^2}\sum_{k=m+1}^\infty \frac{\sqrt{k!}}{(m+1)!(k-m-1)!}.
	\end{equation}
\end{lemma}

\begin{proof}
	Using the truncation error bound of \cite{Mathias}, we have 
	$$\left\|\Delta(A)-\sum_{k=1}^m a_k A_k\right\|\leq \frac{1}{(m+1)!}\max_{s\in [0,1]} \left\|A^{m+1}\Delta^{(m+1)}(sA)\right\|.$$
	Since the $m$-th derivative of $\Delta(z)$ is given by 
	$$\Delta^{(m)}(z)=\sum_{k=m+1}^{\infty} k(k-1)\ldots (k-m+1) a_kz^{k-m},$$
	we can write
	$$\Delta^{(m+1)}(sA)=\sum_{k=m+1}^{\infty} k(k-1)\ldots (k-m+1) a_ks^{k-m-1}A^{k-m-1},$$ 
	yielding
	$$A^{m+1}\Delta^{(m+1)}(sA)=\sum_{k=m+1}^{\infty} k(k-1)\ldots (k-m+1) a_ks^{k-m-1}A^{k}.$$ 
	Taking norms and accounting that $s\leq 1$ and $\|A\|\leq 1$,  
	$$\left\|A^{m+1}\Delta^{(m+1)}(sA)\right\|\leq \sum_{k=m+1}^{\infty} k(k-1)\ldots (k-m+1) |a_k|.$$ 
	Attending that 
	$$|a_k| \lesssim \frac{4}{\pi^2\sqrt{\Gamma(n+1)}},$$
	(see \cite{Bourguet}), the relationship (\ref{trunc-bound}) follows.
	\qed
\end{proof}

For convenience, let us change the index $k$ in the series in the right-hand side of (\ref{trunc-bound}) to $k=p+m$. Then the series can be rewritten as
\begin{equation}\label{series-bound}
\frac{4}{\pi^2}\sum_{p=1}^\infty \frac{\sqrt{(p+m)!}}{(m+1)!(p-1)!}.
\end{equation}
By the d'Alembert ratio test, we can easily show that (\ref{series-bound}) is convergent. Indeed, denoting 
$$b_p:=\frac{\sqrt{(p+m)!}}{(m+1)!(p-1)!},$$
one has
$$\lim_{p\rightarrow\infty} \frac{b_{p+1}}{b_p}=0.$$
The exact value of (\ref{series-bound}) is unknown and so we will work with estimates. We have approximated the sum of the series (\ref{series-bound}) in MATLAB, using variable precision arithmetic with $250$ digits, by taking  $p=2000$. Assuming that $\|A\|\leq 1$, we have found that for $m=33$, one has
$$\Delta(A)\approx \sum_{k=1}^{33} a_k A_k,$$
with a truncation error of about $1.1294\times 10^{-17}$. This means that $33$ terms of the reciprocal gamma function series is a reasonable choice if the calculations are performed in IEEE double precision arithmetic environments.  

For the more general case when $\|A\|>1$, our strategy is to combine (\ref{reciprocal-series-matrix}) with the Gauss multiplication formula 
\begin{equation}\label{gauss-mult-matrix}
\Delta(A)=(2\pi)^{\frac{r-1}{2}}\,r^{I/2-A}\,\prod_{k=0}^{r-1}\Delta\left(\frac{A+kI}{r}\right),
\end{equation}
where $r$ is a positive integer.

Given a positive real number $\mu$, we aim to find a positive integer $r$ for which
$$
\rho\left(\frac{A+(r-1)I}{r}\right) \leq \mu,
$$
or, equivalently,
\begin{equation}\label{eq-mu}
\rho\left(A+(r-1)I\right) \leq r\mu. 
\end{equation}
This guarantees that the arguments of the reciprocal gamma function arising in the right-hand side of (\ref{gauss-mult-matrix}) are matrices with eigenvalues lying on the circle with centre at the origin and radius $\mu$. Hence, if $\mu$ is small enough, taking an $r$ satisfying (\ref{eq-mu}) and a suitable number of terms $m$ in (\ref{reciprocal-series-matrix}) will give an approximation to $\Delta(A)$ with good accuracy. More details are given in the following.

Since $\rho(A+B)\leq \rho(A)+\rho(B)$, for two given commuting matrices \cite[p. 117]{Horn13}, we know that
$$\rho\left(A+(r-1)I\right) \leq \rho(A)+(r-1).$$
Finding the smallest $r$ such that 
\begin{equation}\label{eq-mu-2}
\rho(A)+(r-1)\leq r\mu,
\end{equation}
yields an $r$ satisfying (\ref{eq-mu}). Hence, providing that $\rho(A)>1$ and $\mu>1$, one can take
$$r=\left\lceil \frac{\rho(A)-1}{\mu-1} \right\rceil.$$
What is difficult in this approach is to find optimal values for $\mu$ and $r$ in order to minimize the number operations involved, while guaranteeing a small error. Based on several tests we have carried out (not reported here), a reasonable choice for working in IEEE double precision arithmetic seems to be $\mu=3$ and $m=50$ (number of terms taken in (\ref{reciprocal-series-matrix})).  

The computation of the gamma function by means of the series of its reciprocal is summarized in Algorithm \ref{alg-reciprocal-1}.

\begin{algorithm}\label{alg-reciprocal-1}
	{\rm This algorithm approximates $\Gamma(A)$ , where $A\in \mathbb{C}^{n\times n}$ is a non-singular matrix with no negative integers eigenvalues, by the reciprocal gamma function series combined with the Gauss multiplication formula. Assume that the coefficients $a_1,\ldots,a_{50}$ in (\ref{reciprocal-series-matrix}) are available.
		\begin{enumerate}
			\item $\mu=3$;
			\item \texttt{if} $\rho(A)\leq \mu$
			\item $\quad \widetilde{\Delta} = \sum_{k=1}^{50} a_kA^k$; 
			\item $\quad \Gamma(A)\approx (\widetilde{\Delta})^{-1}$;
			\item \texttt{else}
			\item $\quad$ Compute $r=\left\lceil \frac{\rho(A)-1}{\mu-1} \right\rceil$;
			\item  $\quad$ $\widetilde{\Delta}=\sum_{k=0}^{50} a_k\left(\frac{A}{r}\right)^k$;
			\item $\quad$ \texttt{for} $p=1:r-1$
			\item $\quad\quad$ Compute $\widetilde{\Delta}=\widetilde{\Delta}\sum_{k=0}^{50} a_k\left(\frac{A+pI}{r}\right)^k$;
			\item $\quad$ \texttt{end}
			\item $\quad$ $\widetilde{\Delta}=(2\pi)^{\frac{r-1}{2}}\,r^{0.5\,I-A}\,\widetilde{\Delta}$;
			\item $\quad \Gamma(A)\approx (\widetilde{\Delta})^{-1}$;
			\item \texttt{end}
	\end{enumerate}}
\end{algorithm}

Many techniques for evaluating the matrix polynomials in steps 3 and 9 of the previous algorithm are available \cite[Sec.4.2]{Higham}. One of the most popular is the Horner method which requires $m-1$ matrix multiplications for a polynomial of degree $m$; it is implemented in \texttt{polyvalm} of MATLAB and was used in our implementations of the algorithms whose results will be presented in Section \ref{experiments}. More sophisticated, but less expensive techniques, such as the Paterson \& Stockmeyer method, could be obviously used.

\subsection{Schur-Parlett Approach}\label{schur-parlett}

We start by revisiting the Schur decomposition and the block-Parlett recurrence. This block recurrence is an extension of the original Parlett method proposed in \cite{Parlett}. For additional information, we refer the reader to \cite{Davies} and \cite[Ch. 9]{Higham}.  

Given $A\in \mathbb{C}^{n\times n}$, the Schur decomposition states that there exists a unitary matrix $U$ and a upper triangular matrix $T$ such that $\,A=UTU^\ast,$ with $T$ displaying the eigenvalues of $A$ in the diagonal. Hence, assuming that $A$ is nonsingular with no negative integers eigenvalues, 
$$\Gamma(A)=U\,\Gamma(T)\, U^\ast,$$ meaning that the evaluation of $\Gamma(A)$ may be reduced to the computation of the gamma function of a triangular matrix. 
Let
\begin{equation} \label{matrizT}
T=\left[
\begin{array}{cccc}
T_{11}&T_{12}&\cdots&T_{1p}\\
0&T_{22}&\cdots&T_{2p}\\\vdots&\ddots&\ddots&\vdots\\0&\cdots&0&T_{pp}
\end{array}
\right]\ \in\mathbb{C}^{n\times n},\ \sigma(T)\cap \mathbb{Z}_0^-=\emptyset,
\end{equation} be written as a $(p\times p)$-block-upper triangular, with the blocks $T_{ii}\ (i=1,\cdots,p)$ being square with no common eigenvalues, that is, 
\begin{equation}\label{espectroT}
\sigma(T_{ii})\cap\sigma(T_{jj})=\emptyset,\ i,j=1,\ldots,p,\ i\neq j.
\end{equation}
Leu us denote 
\begin{equation}\label{L1}
G:=\Gamma(T)=\left[
\begin{array}{cccc}
G_{11}&G_{12}&\cdots&G_{1p}\\
0&G_{22}&\cdots&G_{2p}\\\vdots&\ddots&\ddots&\vdots\\0&\cdots&0&G_{pp}
\end{array}
\right], \end{equation} where $G_{ij}$ has the same size as $T_{ij}\ (i,j=\,\ldots,p)$. Recall that the diagonal blocks of $G$ are given by
$G_{ii}=\Gamma(T_{ii})$. Since $GT=TG$, it can be shown that 
\begin{equation}\label{parlett2}
G_{ij}T_{jj}-T_{ii}G_{ij}=T_{ij}G_{jj}-G_{ii}T_{ij}+
\sum^{j-1}_{k=i+1}(T_{ik}G_{kj}-G_{ik}T_{kj})\quad i<j.
\end{equation}
To find the blocks of $G$, we start by computing the blocks on diagonal $G_{ii}=\Gamma(T_{ii})$. This can be done by algorithms \ref{alg-lanczos-2}, 
\ref{alg-spouge-2} or \ref{alg-reciprocal-1}. In terms of computational cost, there is the advantage of $T_{ii}$ being triangular matrices. 

Once the blocks $G_{ii}$ have been computed, we can use successively (\ref{parlett2}) to approximate the remaining blocks of $G$. Note that for each $i<j$, the identity (\ref{parlett2}) is a Sylvester equation of the form 
\begin{equation} \label{sylvester}
XM-NX=P,
\end{equation} where $M$, $N$ and $P$ are known square matrices and $X$ has to be determined. Equation 
(\ref{sylvester}) has a unique solution if and only if $\sigma(M)\cap\sigma(N)=\emptyset$. Hence, the block-Parlett method requires the solution of several Sylvester equations with a unique solution. Recall that $\sigma(T_{ii})\cap\sigma(T_{jj})=\emptyset,\ i\neq j,$ is assumed to be valid. For the Parlett method to be successful, the eigenvalues of the blocks $T_{ii}$ and $T_{jj}$, $i\neq j,$ need to be well separated in the following sense:
there exists $\delta>0$ (e.g., $\delta=0.1$), such that 
$$\min\left\{|\lambda-\mu|:\ \lambda,\mu\in\sigma(T_{ii}),\,\lambda\neq \mu \right\} > \delta$$
and, for every eigenvalue $\lambda$ of a block $T_{ii}$ with dimension bigger than $1$, there exists $\mu\in\sigma(T_{ii})$ such that $|\lambda-\mu|\leq \delta$. 

An algorithm for computing a Schur decomposition with ``well separated'' blocks was proposed in \cite{Davies}. It is available in \cite{mftoolbox}. 

Now, if $\Gamma(T_{ii})$ is computed by one of the algorithms \ref{alg-lanczos-2}, \ref{alg-spouge-2} and \ref{alg-reciprocal-1}, a framework combining those algorithms with the Schur-Parlett technique can be given as follows.

\begin{algorithm}\label{alg-schur-parlett}
	{\rm This algorithm approximates $\Gamma(A)$ , where $A\in \mathbb{C}^{n\times n}$ is a non-singular matrix with no negative integers eigenvalues, by Schur-Parlett method combined with the algorithms \ref{alg-lanczos-2}, \ref{alg-spouge-2} or \ref{alg-reciprocal-1}.
		\begin{enumerate}
			\item Compute a Schur decomposition $A=UTU^\ast$ ($U$ is unitary and $T$ upper triangular), where the blocks $T_{ii}$ in the diagonal of $T$ are well separated in the sense defined above; 
			\item Approximate $G_{ii}=\Gamma(T_{ii})$ by one of the algorithms: Algorithm \ref{alg-lanczos-2}, Algorithm \ref{alg-spouge-2} or Algorithm \ref{alg-reciprocal-1};
			\item Solve the Sylvester equations (\ref{parlett2}), in order to compute all the blocks $G_{ij}$, with $i<j$;
			\item $\Gamma(A) \approx UGU^\ast$, where $G=\left[G_{ij}\right]$.			
	\end{enumerate}}
\end{algorithm}

\section{Numerical Experiments}\label{experiments}

We have implemented Algorithm \ref{alg-schur-parlett} in MATLAB, with unit roundoff $u=2^{-53}$, with a set of $15$ matrices, with real and non real entries and sizes ranging from $n=5$ to $n=14$. Some matrices are randomized, but almost of all were taken from MATLAB's gallery (\texttt{lehmer}, \texttt{dramadah}, \texttt{hilb}, \texttt{cauchy}, \texttt{condex}, \texttt{riemann},... ).

The following abbreviations are used:

\begin{tabular}{ll}
	\texttt{par-lanczos}: &  Algorithm \ref{alg-schur-parlett} combined with Algorithm \ref{alg-lanczos-2};\\
	\texttt{par-spouge}: &  Algorithm \ref{alg-schur-parlett} together with Algorithm \ref{alg-spouge-2};\\
	\texttt{par-reciprocal}: &  Algorithm \ref{alg-schur-parlett} with Algorithm \ref{alg-reciprocal-1}.	 
\end{tabular}
\vskip2ex
Figure \ref{figure-3} displays the relative error of algorithms \texttt{par-lanczos}, \texttt{par-spouge} and \texttt{par-reciprocal} for the above mentioned 15 test matrices, compared with the relative condition number of $\Gamma$ at $A$, times the unit round-off: $\cond_{\Gamma}(A)u$. To compute those relative errors, we have considered as ``exact'' matrix gamma function the result obtained by our own implementation of the Lanczos method in MATLAB, using variable precision arithmetic with 250 digits. To compute the relative condition number $\cond_{\Gamma}(A)$, we have implemented Algorithm 3.17 in \cite{Higham}, where the Fr\'echet derivative of $L_\Gamma(A,E)$ at $A$ in the direction of $E$ was given by the $(1,2)$-block of the matrix gamma function evaluated at 
$\left[\begin{array}{rr}
A & E \\
0 & A \\
\end{array}\right].$
Recall that (see \cite[(3.16)]{Higham}) 
$$\Gamma\left(\left[\begin{array}{cc}
A & E \\
0 & A \\
\end{array}\right]\right)=\left[\begin{array}{cc}
\Gamma(A) & L_\Gamma(A,E) \\
0 & \Gamma(A) \\
\end{array}\right],$$
provided that $\Gamma$ is defined at $A$.

\begin{figure}[ht]
	\centering
	\hspace*{-1cm}\includegraphics[width=16cm]{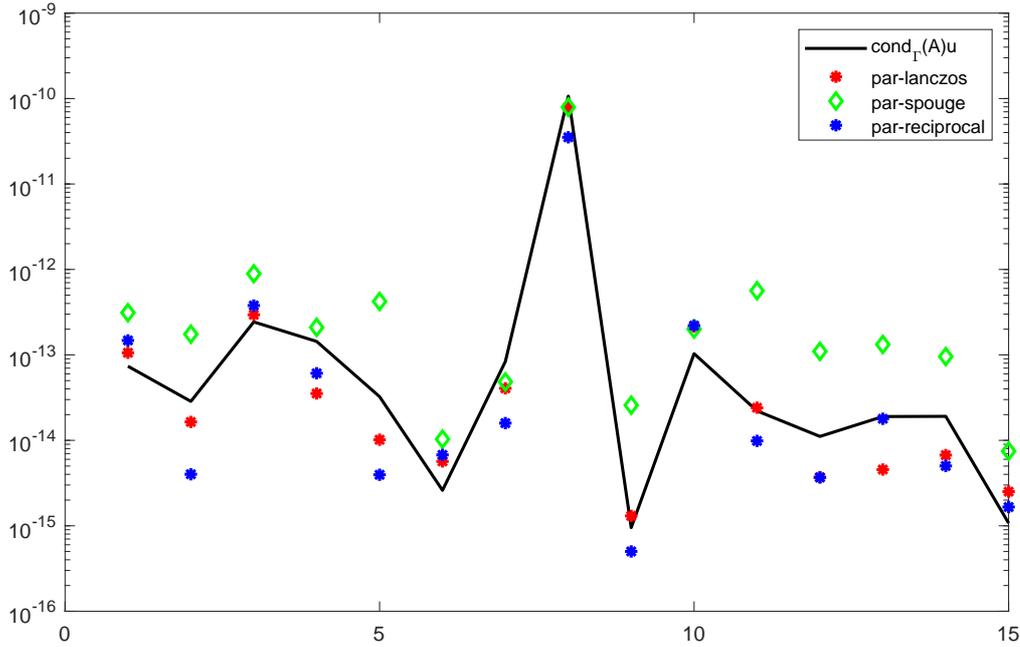}
	\caption{\small Relative error of the three proposed methods for 15 matrices together with the relative condition number of $\Gamma(A)$ times the unit roundoff of MATLAB.}
	\label{figure-3}
\end{figure}

In Figure \ref{figure-3}, by comparison of the relative errors with the solid line corresponding to $\cond_{\Gamma}(A)u$, we observe that algorithms \texttt{par-lanczos} and \texttt{par-reciprocal} perform in a more stable fashion than \texttt{par-spouge}. In terms of accuracy, \texttt{par-reciprocal} gives the best results, but it is the method that has the highest computational cost. However, it seems to be a very promising method because it is rich in matrix-matrix products, which turns it suitable for parallel architectures (note that due to the Schur-Parlett approach, such products are among matrices with small size if compared with the size of $A$) and it can be adapted to high precision computations by increasing the number of terms in the series or by reducing the parameter $\mu$. It can also be implemented without the Schur-Parlett approach, because it works for matrices having simultaneously eigenvalues with positive and negative real parts. Recall that, without using  the Schur-Parlett method, Lanczos and Spouge approximations require that $A$ has just eigenvalues on the righ-half plane or on the left-half plane. 

In any way, our experiments suggest that \texttt{par-lanczos} is apparently the one that states the better compromise between accuracy and computational cost in computations using IEEE double precision arithmetic.

\section{Conclusions}\label{conclusions}

Theoretical issues related with the matrix gamma function have attracted the interest of researchers due to its applications in certain matrix differential equations and its connections with other important functions such as matrix beta and Bessel functions. However, as far as we know, we are the first to provide a thorough investigation on the numerical computation of $\Gamma(A)$.   Three methods have been analysed: Lanczos method, Spouge method and a method based on a Taylor expansion of the reciprocal gamma function combined with the Gauss multiplication formula. All of them have been implemented together with the Schur-Parllet method and tested with several matrices. The deviation of the relative error from $\cond_{\Gamma}(A)u$ is bigger in Spouge method, which lead us to conclude that Lanczos and reciprocal gamma approximations are preferable. New bounds for the norm of the matrix gamma function and its perturbations, and for the truncation errors arising in the approximation methods have been proposed as well.

\end{document}